\documentclass{article}
\usepackage{aarons_packages}
\title{A generalisation of the Poincaré-Hopf Theorem.}
\author{Aaron Pim}
\usepackage{biblatex}  
\begin{document}

\maketitle
\begin{abstract}
    The Poincaré-Hopf Theorem is a conservation law for real-analytic vector fields, which are tangential to a closed surface (such as a torus or a sphere). The theorem also governs real-analytic vector fields, which are tangential to surfaces with smooth boundaries; in these cases, the vector field must be pointing in the outward normal direction along the boundary \cite{Poincare_hopf,Poincare_Hopf_theorem}. In this paper, I will generalise the Poincaré-Hopf Theorem for real-analytic vector fields that are tangential to surfaces with piecewise smooth boundaries, and not parallel to the outward normal of the boundary.
\end{abstract}
\section{Introduction}
    A well-known result in topology is the \textit{Hairy Ball Theorem}, that, in layman's terms, states that you cannot comb the hairs on a tennis ball in one smooth motion without at least one tuft of hair. The Hairy Ball Theorem is a specific subcase of the Poincaré-Hopf Theorem, which is a conservation law for real-analytic vector fields, that are tangential to a closed surface (such as a torus or a sphere). The theorem also governs real-analytic vector fields, which are tangential to surfaces with smooth boundaries; in these cases, the vector field must be pointing in the outward normal direction along the boundary \cite{Poincare_hopf,Poincare_Hopf_theorem}.
    
    A singular value of a vector field may be characterised by a real value known as the \textit{degree}, which represents the number of times the vector field rotates counterclockwise around the singularity \cite{Brouwer_Degree, Winding_Number}. The index of a vector field is equal to the sum of the degree's. If the singular value is positioned in the interior of the domain, then the index may take values only in the integers; however, a singular value located on the boundary may take any real value \cite{Virga}. 
    
    The Poincaré-Hopf Theorem states that the sum of the indices of a vector field, that is tangential to a given surface, is equal to the Euler characteristic of the surface. The Euler characteristic is a topological invariant (constant) that describes the shape of a surface. For example, a sphere is a closed surface that has Euler characteristic of two, and therefore the index of the vector field must equal 2, implying that a continuous configuration does not exist.
    %In the context of liquid crystals, a vector field that is tangential to a surface may represent the director of a \textit{nematic shell} \cite{Nematic_Shells}. A nematic shell is a thin film of liquid crystal that coats a smooth, rigid surface \cite{Analysis_Nematic_Shells}. 
    
    %The focus of the thesis is the modelling of defects in smooth vector field. These defects can be observed as \textit{knots} in the polarisation of light after it has propagated through the liquid crystal \cite{Oleg_pictures}. 
    
    %As previously mentioned, the defect strength may be defined in terms of the index of the vector field in a neighbourhood of the point defect. In the thesis, I shall use the term index when discussing the behaviour of the vector field in a neighbourhood or along the boundary of a simple region; I shall use the term strength when discussing the behaviour of the vector field at a point. 
    
    The author M. Morse derives a formula about the sum of the indices of the singular points of vector fields that are subject to generalised boundary conditions \cite{Morse_OG}. He utilised the notion of the Brouwer degree between manifolds to relate the Euler characteristic of a surface to the degree of the zeros of a smooth vector field and the degree of the boundary data. I will now present the result, using the notation of Canevari, Segatti and Veneroni \cite{Morse_LC}. 
    
    Let $S \subset \mathbb{R}^d$ be a compact connected orientable n-submanifold with boundary, $M \subset \mathbb{R}^d$ be a connected orientable manifold without boundary, and $\mathbf{U}:S \rightarrow M$ be a smooth map. The point $\underline{p}\in M \setminus \mathbf{U}(\partial S)$ is a regular value of $\mathbf{U}$ if $\text{det}(\nabla \mathbf{U})(\underline{s}) \neq 0$ for all $\underline{s} \in \mathbf{U}^{-1}(\underline{p})$. The Brouwer degree of the map $U$, at the point $\underline{p}$, \cite{Brouwer_Degree} is given by 
    \begin{equation}\nonumber
        \text{deg}(\mathbf{U}, S, \underline{p}) := \sum\limits_{\underline{s}\in \mathbf{U}^{-1}(\underline{p})}\text{sign}\left(\text{det}\left(\nabla \mathbf{U}\right) \right)(\underline{s}).
    \end{equation}
    As $\mathbf{U}$ is smooth, the only points where singular values may arise are the zeroes of the vector field. Therefore, the index may be given by $\text{deg}(\mathbf{U}, S, \underline{0})$. However, if the dimension of $S$ is less than the dimension of $M$, as is the case when modelling the director of a nematic shell, an appropriate projection map $\phi:S \rightarrow \mathbb{R}^d$ is implemented such that the index
    \begin{equation}\nonumber
        \text{ind}(\mathbf{U},S) := \text{deg}(\mathbf{U}\circ \phi^{-1}, \phi(S), \underline{0}),
    \end{equation}
    is well-defined. To account for the generalised boundary conditions, Morse considers the field $P_{\partial S}\mathbf{U}$ given by
    \begin{equation}\nonumber
        (P_{\partial S}\mathbf{U})(\underline{s}) := \text{proj}_{T_{\underline{s}}\partial S}\mathbf{U}(\underline{s}), \quad \forall \underline{s}\in \partial S,
    \end{equation}
    where $T_{\underline{s}}\partial S$ is the space of vector fields which are tangential to $\partial S$ at the point $\underline{s}\in \partial S$. Let the space $\partial_{-}S(\mathbf{U})\subset \partial S$ be defined by 
    \begin{equation}\nonumber
        \partial_{-}S(\mathbf{U}) := \{\underline{s} \in \partial S, \ (\mathbf{U}\cdot \nu)(\underline{s})\},
    \end{equation}
    where $\nu(\underline{s}) \in T_{\underline{s}}S$ is the outward pointing unit boundary normal of $S$.
    \begin{proposition}{Morse's index formula}
        If $\mathbf{U}$ is a continuous vector field over $S$ satisfying $\underline{0}\notin \mathbf{U}(\partial \partial S)$, with finitely many zeros and if $P_{\partial S}\mathbf{U}$ has finitely many zeros, then 
        \begin{equation}\label{chapter 2: Morse index formula}
            \text{ind}(\mathbf{U},S) + \text{ind}(P_{\partial S}\mathbf{U},\partial_{-}S) = \chi(S).
        \end{equation}
    \end{proposition}
    The terms $\text{ind}(\mathbf{U},S)$ and $\text{ind}(P_{\partial S}\mathbf{U},\partial_{-}S)$ may be interpreted as the number of complete rotations that the vector $\mathbf{U}$ along $\partial S$, relative to a fixed co-ordinate vector and the boundary tangent respectively. 
    
    The Morse index formula is a generalisation of the Poincaré-Hopf Theorem by permitting vertices along the boundary of $\partial S$ and boundary data which is not parallel to the boundary normal. I will generalise the Morse index formula by permitting finitely many zeroes along the boundary of $\partial S$. 
%I shall now discuss the choice of notation for the thesis. In topology, subscripts and superscripts denote the covariance and contravariance of tensors, respectively \cite{Contravariance_and_Covariance}. This convention is used to determine how a tensor changes with respect to a change of basis; however, this notation is not without its flaws. For example, when defining the norm squared of a column vector, $\mathbf{a}:= \{a_i\}_{i=1}^N$ I cannot simply write $|\mathbf{a}|^2 = a^ia_i$ as the contravariance of $\mathbf{a}$ is undefined. Therefore, I must either use a Kronecker delta $|\mathbf{a}|^2 = a_i\delta^{ij}a_j$ or define what is meant by $a^i$ (which is implicitly $\mathbf{a}^T$). This can be rather cumbersome and can become an issue when calculating an inner product of multiple different indices. Thus, I have elected not to use this convention, as it will be clear from the context the meaning of my notation.
\section{Assumptions and definitions}
This section shall consist of a series of subsections, each one dedicated to defining a key concept which shall be used in the derivation of the generalised Poincaré-Hopf Theorem. I refer the reader to ``Differential Geometry of Curves and Surfaces''by Manfredo do Carmo \cite{Differential_Geo}, for more details regarding the notation and definitions in this paper.
\subsection{Assumptions on the regularity of the surface}\label{Surface assumptions}
    Let $S \subset \mathbb{R}^3$ be an orientable two-dimensional manifold embedded in three dimensions. It is assumed that there exists an open, connected and bounded set $\Omega \subset \mathbb{R}^2$, and a bijective, smooth chart $\mathbf{x}^{-1}: S\rightarrow \Omega$ with inverse $\mathbf{x} \in C^2(\overline{\Omega}, \overline{S})$. The natural trihedron of the surface is denoted $\{\mathbf{e}^1,\mathbf{e}^2, \mathbf{N}\}$, and is explicitly given by:
  	\begin{equation}\label{chapter 2: natural trihedron defn}
  	  \mathbf{x}^i := \dfrac{\partial \mathbf{x}}{\partial \omega_i}, \quad \mathbf{e}^i := \dfrac{\mathbf{x}^i}{|\mathbf{x}^i|}, \quad
  	  \mathbf{N}:= \dfrac{\mathbf{x}^1\times \mathbf{x}^2}{\left| \mathbf{x}^1\times \mathbf{x}^2 \right|}, \text{ for } \underline{\omega}:=(\omega_1,\omega_2) \in \Omega.
    \end{equation}
    I shall assume that the surface $S$ is such that the natural trihedron forms an orthonormal basis:    \begin{equation}\label{chapter 2: assumption of orthonormality}
        \mathbf{x}^1 \cdot \mathbf{x}^2 (\underline{\omega}) = 0 \quad \forall \underline{\omega} \in \Omega.
    \end{equation}
    Additionally, I will assume that the Gaussian curvature, denoted $K:\Omega \rightarrow \mathbb{R}$, is of the class $L^1(S)$. This is to ensure that the Gauss-Bonnet theorem is applicable to the surface. The proof of the generalised Poincaré-Hopf Theorem, utilises the Gaussian curvature, which may be expressed in terms of the derivatives of the normalised curvilinear basis vectors. An illustration of the curvilinear basis may be seen in figure (\ref{fig: Natural Trihedron}).
    \begin{figure}[thbp]
    \centering
    \includegraphics[width = 0.5 \textwidth]{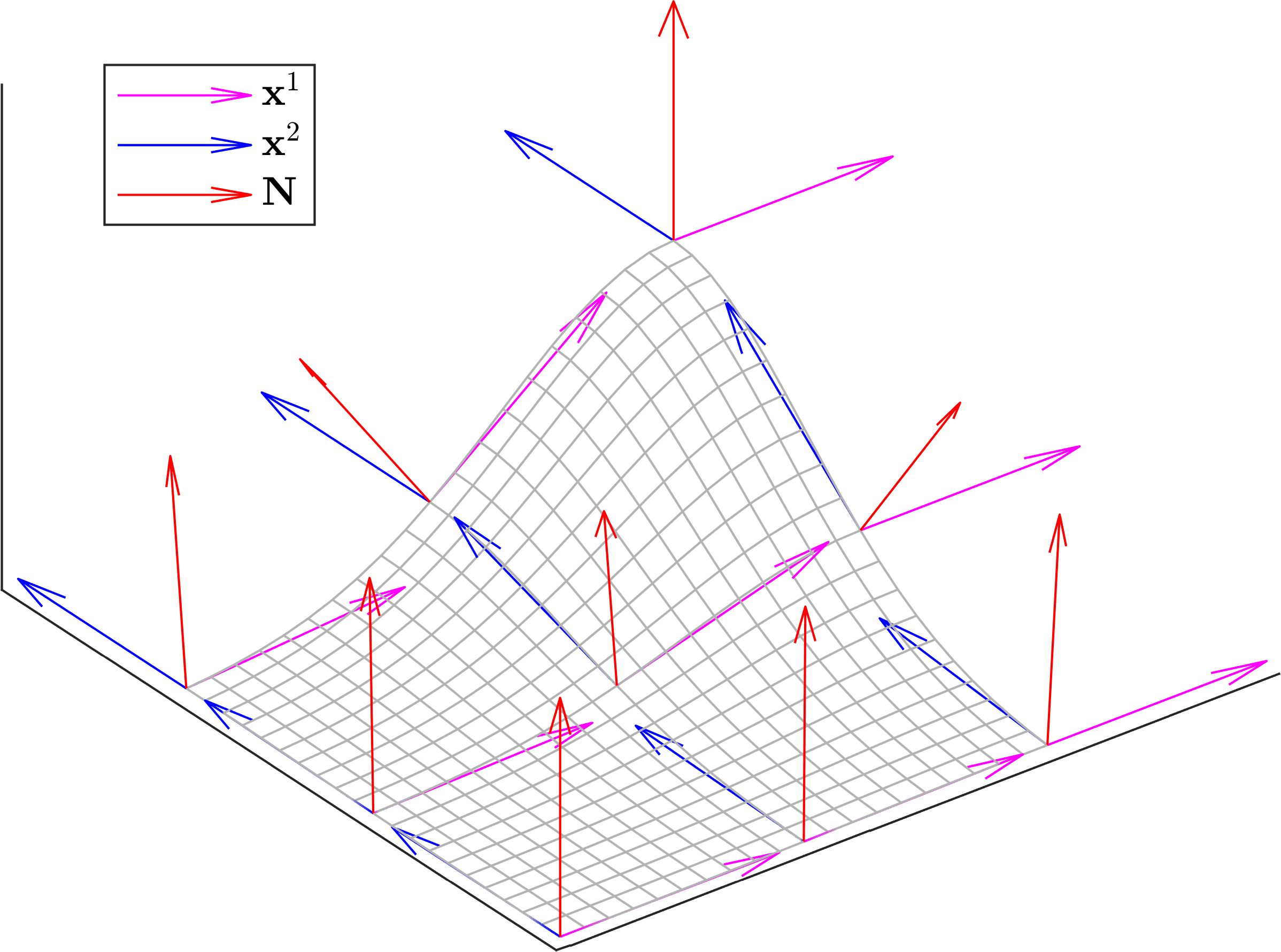}
    \caption{An illustration of the natural trihedron, note that despite the basis vectors changing with respect to the position, they remain orthogonal.}
    \label{fig: Natural Trihedron}
    \end{figure}
    \begin{lemma}
        If the parametrisation is orthogonal, the Gaussian curvature of the surface $S$ is given by
        \begin{equation}\label{Gaussian Curvature}
            K := \dfrac{1}{|\mathbf{x}^1||\mathbf{x}^2|}\left(\dfrac{\partial}{\partial \omega_2}\left(\dfrac{\partial \mathbf{e}^1}{\partial \omega_1}\cdot \mathbf{e}^2 \right)+\dfrac{\partial}{\partial \omega_1}\left(\dfrac{\partial \mathbf{e}^2}{\partial \omega_2}\cdot \mathbf{e}^1 \right)\right).
        \end{equation}
    \end{lemma}
    \begin{proof}
        See \cite{Differential_Geo}.
    \end{proof}
    \begin{remark}
        The orthogonality of the basis vector implies that
        \begin{equation}\nonumber
            \dfrac{\partial \mathbf{e}^2}{\partial \omega_j}\cdot \mathbf{e}^1 = -\dfrac{\partial \mathbf{e}^1}{\partial \omega_j}\cdot \mathbf{e}^2.
        \end{equation}
        Consequently, the assumption on the regularity of the Gaussian curvature may be expressed in the following manner,
        \begin{equation}\label{Assumption on the Christoffel symbols}
            K|\mathbf{x}^1||\mathbf{x}^2|=\dfrac{\partial F_1}{\partial \omega_2}-\dfrac{\partial F_2}{\partial \omega_1} \in L^1(\Omega), \text{ where }F_j:= \dfrac{\partial \mathbf{e}^1}{\partial \omega_j}\cdot \mathbf{e}^2.
  	    \end{equation}
    \end{remark}
    Surfaces of revolution, with a sufficiently smooth profile, are examples of surfaces which fulfil the above assumptions. 
\subsection{Assumptions on the regularity of the boundary}\label{Boundary assumptions}
    It is assumed that the boundary of the surface $S$, denoted $\partial S \subset \mathbb{R}^3$, may be expressed as the union of $M \in \mathbb{N}$ disjoint closed components, denoted $\partial S_1, \ldots, \partial S_M$. The case $M=0$ corresponds to a closed surface, in which the original Poincaré-Hopf Theorem may be applied. For $i=1,\ldots , M$, each boundary component $\partial S_i$ is parametrised by the Lipschitz continuous function $\underline{\gamma}_i:[0,l_i]\rightarrow \partial S_i$, that satisfies the following conditions:
    \begin{itemize}
        \item The constant $l_i>0$ is the arclength of the curve $\partial S_i$ and the function $\underline{\gamma}_i$ is parametrised by its arclength.
        \item The parametrisation satisfies $\underline{\gamma}_i(0)=\underline{\gamma}_i(l_i)$.
        \item There exists a finite number of values $0=:l^0_i< \ldots< l^n_i:=l_{i}$, which are referred to as the \textit{vertices} of $\partial S_i$, where the curve is not twice differentiable: \begin{equation}\label{chapter 2: boundary parametrisation}
            \underline{\gamma}_i \in \bigcup\limits_{j=1}^{n_i}C^2\left(\left(l^{j-1}_i,l^j_i\right),\partial S_i\right).
        \end{equation}
    \end{itemize}
    The number of vertices of $\partial S$ is denoted $\text{V}_{\partial S} \in \mathbb{N}\cup \{0\}$. Additionally, the exterior angles, of $\partial S$, are denoted $\tau_1, \ldots, \tau_{\text{V}_{\partial S}} \in (-\pi,\pi)\setminus \{0\}$. Note that the exterior angle is $\pi$ minus the interior angle, an illustration of the interior angle can be found in figure (\ref{fig:exterior angle three graphs}).
    \begin{figure}[htbp]
       \centering
       \begin{subfigure}[b]{0.3\textwidth}
         \centering
         \includegraphics[width=\textwidth]{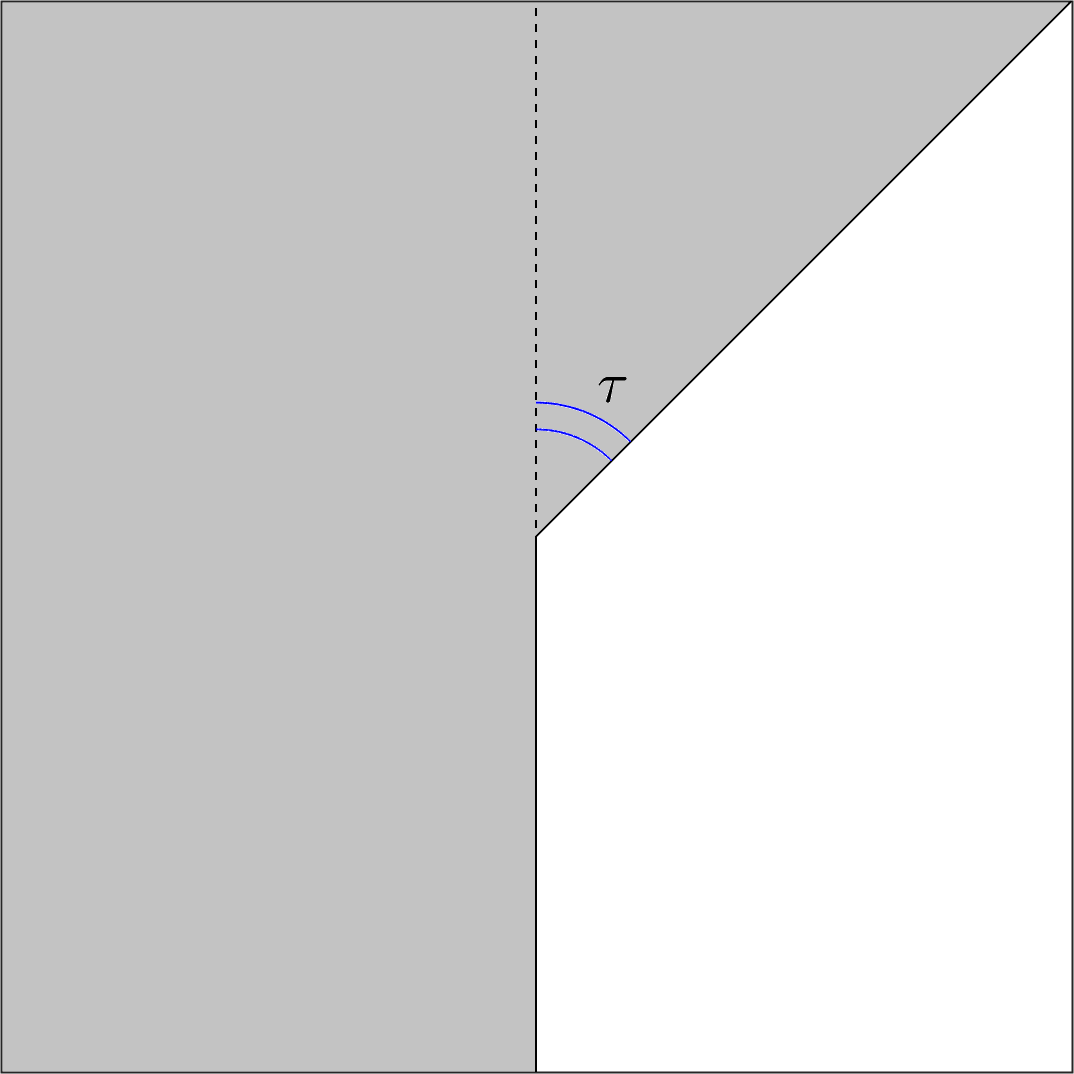}
         \caption{An illustration of $\tau < 0$, which corresponds to a reflex interior angle.}
         \label{subfig:tau<0}
       \end{subfigure}
       \hfill
       \begin{subfigure}[b]{0.3\textwidth}
         \centering
         \includegraphics[width=\textwidth]{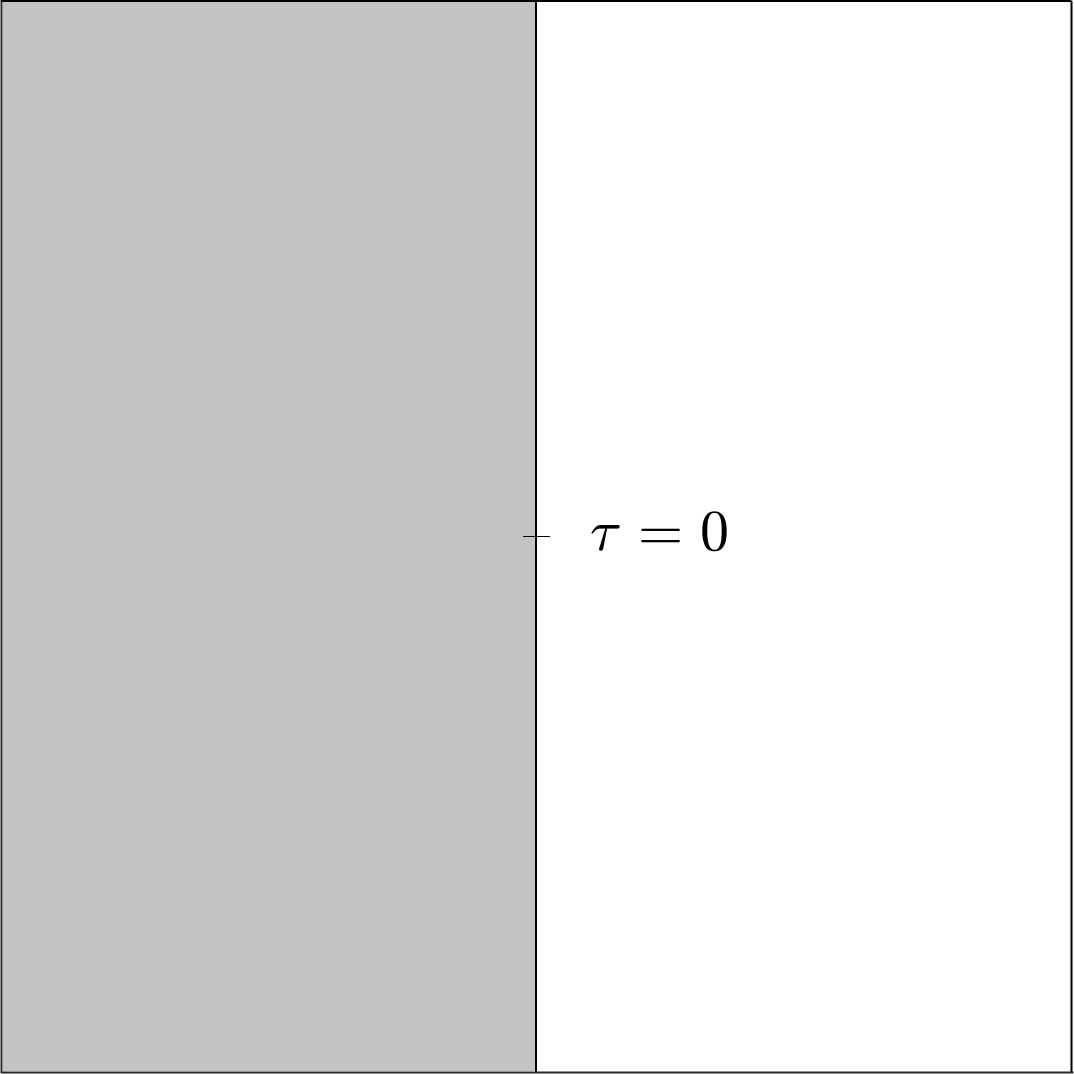}
         \caption{An illustration of $\tau = 0$, which corresponds to a flat interior angle.}
         \label{subfig:tau=0}
       \end{subfigure}
       \hfill
       \begin{subfigure}[b]{0.3\textwidth}
         \centering
        \includegraphics[width=\textwidth]{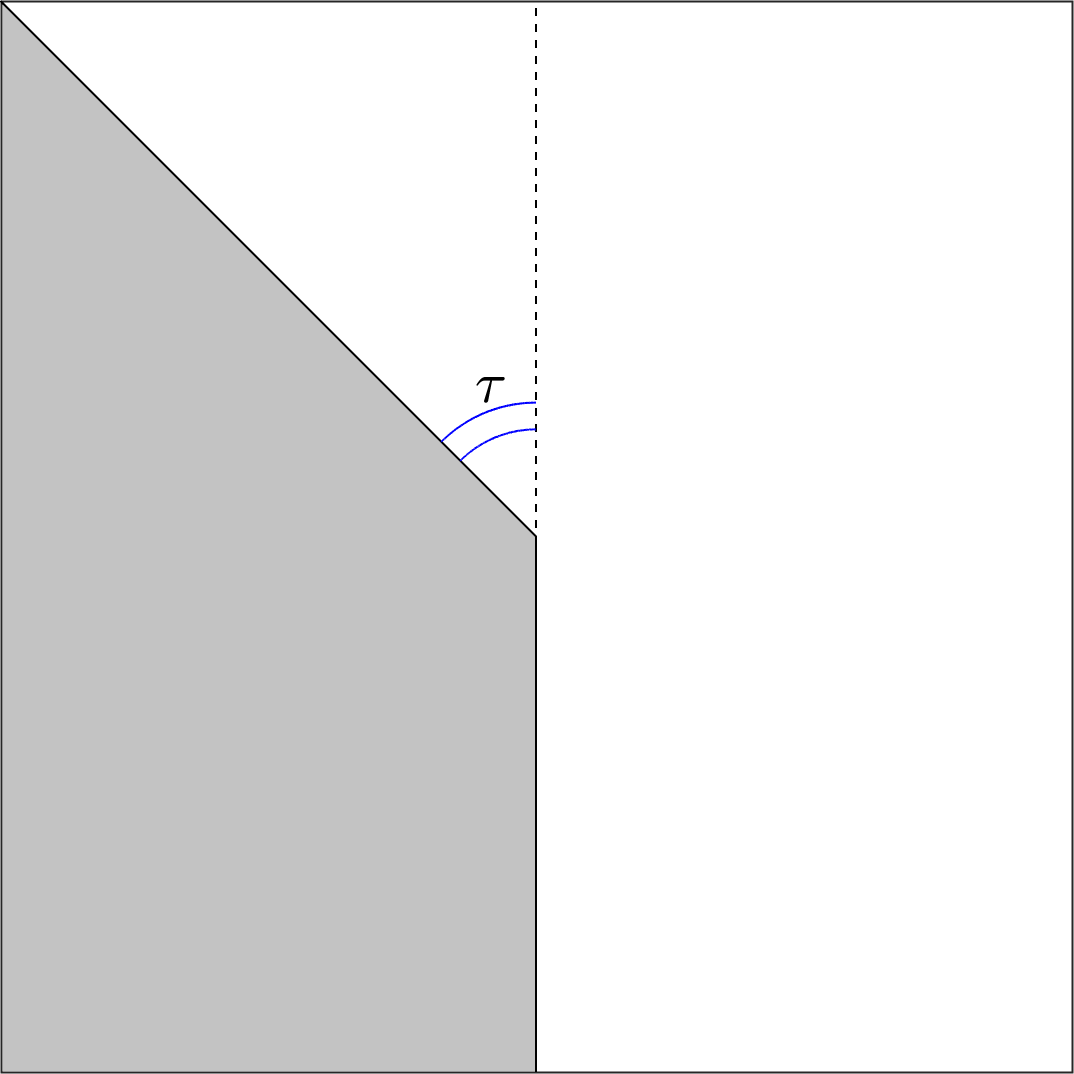}
         \caption{An illustration of $\tau > 0$, which corresponds to an acute or obtuse interior angle.}
         \label{subfig:tau>0}
       \end{subfigure}
        \caption{Above is a series of diagrams showing the behaviour of the exterior angle for varying signs. The sign of the exterior angle indicates the type of interior angle. If the exterior angle is positive then the interior angle is acute or obtuse, if the exterior angle is zero then the interior angle is flat, and if the exterior angle is negative then the interior angle is a reflex angle. Additionally, several equations (including the generalised Poincaré–Hopf Theorem) are explicitly dependent on the exterior angle.}
        \label{fig:exterior angle three graphs}
    \end{figure}   
    
    It is assumed that the geodesic curvature, denoted $kg:W^{2,1}([0,l],S)\rightarrow L^1([0,l])$, of the boundary component $\underline{\gamma}_i$ of class $L^1([0,l_i])$,
    \begin{equation}\label{chapter 2: geodesic curvature defin}
        kg\left(\underline{f}\right) := \underline{f}''\cdot \left(\left(\mathbf{N}\circ \mathbf{x}^{-1}\circ \underline{f}\right)\times \underline{f}'\right), \quad kg\left(\underline{\gamma}_i\right) \in L^1\left(\left(0,l_i\right)\right), \quad \forall i = 1, \ldots M.
    \end{equation}
     
\subsection{Definition of the triangulation of a surface}
    In this paper, there are two generalisations of the Poincaré-Hopf Theorem, the first for simple domains and the second for non-simple domains. The proof of the theorem for non-simple domains relies on the theorem for simple domains, where I partition the surface $S$ into a set of simple subdomains and then I apply the previous theorem to each subdomain. The particular partition that I require is known as a triangulation. In this subsection, I will define what I mean by a triangulation.
    \begin{definition}
        An open subregion $S_R\subset S$ is simple if it is homeomorphic to a disk. A simple subregion $T \subset S$ is a triangle if it has exactly three vertices (with non-zero exterior angles). I let:
        \begin{itemize}
            \item The set of vertices of the triangle $T$ be denoted $\mathcal{V}_T:=\{v^1_T, v^2_T, v^3_T\}$ (where $v^i_T\in \partial T$).
            \item The set of edges be denoted $\mathcal{E}_T:=\{E^1_T, E^2_T, E^3_T\}\subset T$.
        \end{itemize}
        I require that the set of edges and vertices satisfy $E^i_T \cap E^j_T = \emptyset$ for $i \neq j$, and $\partial T \setminus \mathcal{E}_T = \mathcal{V}_T$.
    \end{definition}
    I shall now define a triangulation of the surface, and express the Euler characteristic in terms of the triangulation. Those familiar with numerical finite elements methods would be familiar with this specific partition, as I divide the domain into a series of triangular subregions.
    \begin{definition}[Triangulation of the surface]\label{Triangulation of a surface}
        The surface $S$ is assumed to be such that there exists a finite partition, denoted $\{T_i\}_{i=1}^{\text{F}}$ of $S$. Let $\text{F}\in \mathbb{N}$ denote the number of \textit{faces}, each of which is a triangle.
        
        It is assumed that the edge of a triangle does not contain the vertex of another triangle,
        \begin{equation}\nonumber
            \partial T_i \cap \partial T_j = \overline{\mathcal{E}_{T_i} \cap \mathcal{E}_{T_j}}, \quad \forall i \neq j.
        \end{equation}
       The number of vertices of the triangulation is denoted $\text{V}$, and the number of edges is denoted $\text{E}$ (ignoring repeated vertices and edges respectively)
        \begin{equation}\nonumber
            \textrm{V}:= \left| \bigcup\limits_{i=1}^{\text{F}} \textrm{V}_{T_i}\right|, \quad \textrm{E}:=\left|\bigcup\limits_{i=1}^{\text{F}} \mathcal{E}_{T_i}\right|.
        \end{equation}
    \end{definition}
    \begin{remark}
        The Euler characteristic of the surface $S$, denoted $\chi \in - \mathbb{N} \cup \{0,1\}$, may be written in terms of the number of faces, edges, and vertices of a triangulation of $S$
        \begin{equation}\label{chapter 2: Euler Characteristic}
            \chi := \textrm{F} - \textrm{E} + \textrm{V}.
        \end{equation}
    \end{remark}
\subsection{Definition of admissible space of vector fields}
    In this subsection, I shall define the admissible class of unit vector fields, which are required to be tangential to the surface $S$ and permit point discontinuities. 
    \begin{definition}[Admissible space of functions]\label{Admissible space of functions}
        The vector field $\mathbf{u}:\overline{S} \rightarrow \mathbb{S}^2$ is of class $\mathcal{A}_{\tan}$ if and only if there exists a real analytic vector field $\mathbf{A} \in C^\infty(\overline{S},\mathbb{R}^3)$, that satisfies the two conditions:
        \begin{itemize}
            \item The zeros of the vector field $A$ are isolated.
            \begin{equation}\label{Definition of Jump Set}
                J_{\mathbf{u}}:= \left\{ \underline{s}\in \overline{S}\left| \mathbf{A}(\underline{s}) = \underline{0} \right. \right\}, \quad \forall \underline{s} \in J_{\mathbf{u}}, \ \exists \delta >0, \text{ s.t. } B_{\delta}(\underline{s})\cap \left(J_{\mathbf{u}}\setminus \{\underline{s}\}\right)= \emptyset.
            \end{equation}
            \item The vector field $A$ is tangential to the surface S.
            \begin{equation}\nonumber
                \mathbf{A}(\mathbf{x}(\underline{\omega}))\cdot \mathbf{N}(\underline{\omega}) = 0, \ \forall \underline{\omega} \in \Omega, \quad \mathbf{u}(\underline{s}):= \dfrac{\mathbf{A}}{|\mathbf{A}|}(\underline{s}), \ \underline{s} \in \overline{S}\setminus J_{\mathbf{u}}.
            \end{equation}
        \end{itemize}
    \end{definition}
    \begin{remark}[Consequences of the first assumption]
        The assumption that the zeros of $\mathbf{A}$ are isolated, implies that $J_{\mathbf{u}}$ is a discrete finite set and that $\mathbf{u} \in C^{\infty}(\overline{S}\setminus J_{\mathbf{u}}, \mathbb{S}^2)$.
    \end{remark}
    \begin{remark}[Consequences of the second assumption]
        The assumption that the vector field $\mathbf{u}$ is tangential to the surface, implies that it may be expressed as a two-dimensional vector field with respect to the curvilinear coordinates from equation (\ref{chapter 2: natural trihedron defn}). For each $\mathbf{u}\in \mathcal{A}_{\tan}$ there exists functions $(u_1,u_2) \in C^{\infty}(\overline{S}\setminus J_{\mathbf{u}}, \mathbb{S}^1)$ such that
        \begin{equation}\label{decomposition of a tangential field}
          	\mathbf{u}(\underline{s}) = u_1(\underline{s})\mathbf{e}^1(\mathbf{x}^{-1}(s)) + u_2(s)\mathbf{e}^2(\mathbf{x}^{-1}(s)), \quad \forall \underline{s} \in \overline{S}\setminus J_{\mathbf{u}}.
        \end{equation}
    \end{remark}
\subsection{Definition of the Christoffel symbols and covariant derivatives}
    In this subsection, I shall define the notation regarding differentiation with respect to a curved basis, reminding readers of the concepts of Christoffel symbols of the second kind, covariant derivatives, and algebraic values.

    For $i,j,k \in \{1,2\}$, let $\Gamma_{ij}^k$  denote the Christoffel symbols of the second kind, and $L_{ij}$ the entries of the second fundamental form respectively:
    \begin{align}\label{chapter 2: curvature definitions}
        \Gamma_{ij}^k &:=\dfrac{\mathbf{x}^{ij}\cdot \mathbf{x}^k}{|\mathbf{x}^k|^2},& L_{ij}&:=\mathbf{x}^{ij}\cdot \mathbf{N},& \text{ where }\quad \mathbf{x}^{ij}&:= \dfrac{\partial^2 \mathbf{x}}{\partial \omega_i \partial \omega_j}.
    \end{align}
    A more intuitive way to understand $\Gamma_{ij}^k$ and $L_{ij}$, is to express $\mathbf{x}^{ij}$ in curvilinear coordinates
    \begin{equation}\nonumber
		\mathbf{x}^{ij}= \Gamma_{ij}^k\mathbf{x}^{k} + L_{ij}\mathbf{N},
    \end{equation}
    and understand that $\Gamma^k_{ik}$ represents the component of $\mathbf{x}^{ij}$ in the $\mathbf{x}^k$ direction; similarly $L_{ij}$ represents the component of $\mathbf{x}^{ij}$ in the $\mathbf{N}$ direction. In the case when the surface is planar, then $L_{ij} = 0$.

    For a given $\mathbf{u} \in \mathcal{A}_{\tan}$ with components $(u_1,u_2) \in C^{\infty}(\overline{S}\setminus J_{\mathbf{u}}, \mathbb{S}^1)$, the partial derivative of the function $\overline{\mathbf{u}}:= \mathbf{u}\circ \mathbf{x}:\Omega \rightarrow \mathbb{S}^2$, with respect to $\omega_i$ is given by
    \begin{equation}\nonumber
        \dfrac{\partial \overline{\mathbf{u}}}{\partial \omega_j} = \left(\dfrac{\partial \overline{u}_i}{\partial \omega_j} +F_j\left(\overline{u}_1 \delta_{2i}-\overline{u}_2 \delta_{1i}\right) \right)\mathbf{e}^i + \left(Z_{ij}\overline{u}_i \right)\mathbf{N},
    \end{equation}
    where
    \begin{equation}\nonumber
        Z_{ij}:= \dfrac{\partial \mathbf{e}^{i}}{\partial \omega_j}\cdot \mathbf{N}.
    \end{equation}
    \begin{remark}
        The components of the vector $\mathbf{F}$, were defined in equation (\ref{Assumption on the Christoffel symbols}). Combining that expression with the above equations, the derivative of the basis vector $\mathbf{e}^j$ may be expressed in terms of $\mathbf{F}$ and $Z$:
        \begin{equation}\nonumber
            \dfrac{\partial \mathbf{e}^i}{\partial \omega_j} = \left(\dfrac{\partial \mathbf{e}^i}{\partial \omega_j}\cdot \mathbf{e}^k \right)\mathbf{e}^k + \left(\dfrac{\partial \mathbf{e}^i}{\partial \omega_j}\cdot \mathbf{N}\right)\mathbf{N} = F_j\left(\delta_{2i}\mathbf{e}^1-\delta_{1i}\mathbf{e}^2 \right)+Z_{ij}\mathbf{N}.
        \end{equation}
    \end{remark}
    Despite $\overline{\mathbf{u}}$ being tangential to the plane, the derivative of $\overline{\mathbf{u}}$ includes a normal component. This motivates the notion of the covariant derivative, which is defined as the projection of the derivative onto the tangent plane of the surface. This principle is illustrated in figure (\ref{fig: co-variant derivative}) and is given by 
    \begin{equation}\label{chapter 2: defn covariant derivative}
    	\dfrac{D\overline{\mathbf{u}}}{d\omega_i} := \dfrac{d \overline{\mathbf{u}}}{d\omega_i} - \left( \dfrac{d \overline{\mathbf{u}}}{d\omega_i}\cdot \mathbf{N}\right) \mathbf{N} = \left(\dfrac{\partial \overline{u}_j}{\partial \omega_i} +F_i\left(\overline{u}_1 \delta_{2j}-\overline{u}_2\delta_{1j} \right) \right)\mathbf{e}^j.
    \end{equation}
    \begin{figure}[htbp]
        \centering
        \includegraphics[width = 0.5\textwidth]{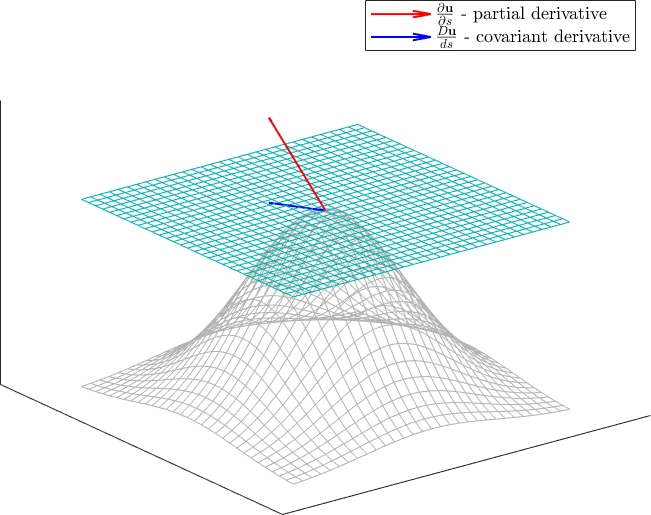}
        \caption{An illustration of the difference between the derivative and co-variant derivative.}
        \label{fig: co-variant derivative}
    \end{figure}
    The covariant derivative may be used to define key concepts, such as the algebraic value and the index of a vector field. To define these concepts I must consider the vector field restricted to a curve within the surface. Let $\underline{\beta} \in C_{\text{lip}}([0,L],S)$ be a curve parametrised by its arclength $L>0$; additionally, let $\hat{\mathbf{u}}:= \mathbf{u}\circ \underline{\beta}$ and $\hat{\mathbf{N}}:= \mathbf{N}\circ \underline{\beta}$, denote the restriction of $\mathbf{u}$ and $\mathbf{N}$ to the curve $\underline{\beta}$.

    The algebraic value, denoted $\left[\dfrac{D\mathbf{u}}{dt} \right]:[0,L] \rightarrow \mathbb{R}$, of the vector field $\hat{\mathbf{u}}$ is given by
    \begin{equation}\label{Definition of Algebraic Value}
        \left[\dfrac{D\mathbf{u}}{dt} \right] := \dfrac{D\hat{\mathbf{u}}}{dt} \cdot \left(\hat{\mathbf{N}}\times \hat{\mathbf{u}}\right).
    \end{equation}
    The algebraic value can be used to define the geodesic curvature of the curve $\underline{\gamma}$, defined in equation (\ref{chapter 2: geodesic curvature defin}), may be expressed in terms of the algebraic value of its derivative 
    $$ kg(\underline{\gamma}) = \left[\dfrac{D \underline{\gamma}'}{dt} \right].$$
    Additionally, the algebraic value can be used to define the angle between two unit tangent vectors. For arbitrary $\mathbf{u},\mathbf{v}\in \mathcal{A}_{\tan}$, the derivative of the angle from $\hat{\mathbf{u}}$ to $\hat{\mathbf{v}}$, denoted $\phi_{\mathbf{u},\mathbf{v}}:[0,L]\rightarrow \mathbb{R}$, is defined in terms of the algebraic value
    \begin{equation}\label{angle between two vector fields}
        \dfrac{d \phi_{\mathbf{u},\mathbf{v}}}{dt} = \left[\dfrac{D\mathbf{v}}{dt}\right]- \left[\dfrac{D\mathbf{u}}{dt}\right].
    \end{equation}
    As a direct consequence of the above equation and the orthogonality of the curvilinear basis, the derivative of the angle from $\hat{\mathbf{u}}$ to $\hat{\mathbf{e}}^1$, denoted $\phi:(0,L)\rightarrow \mathbb{R}$, is given by
    \begin{equation}\label{Equation 1}
        \left[\dfrac{D\mathbf{u}}{dt} \right]= \dfrac{d \left( \mathbf{F}\circ \underline{\beta}\right)}{dt} + \dfrac{d\phi}{dt}.
    \end{equation}
\subsection{Definition of the index of a vector field}
    One of the ways to characterise a tangential vector field, within a simple domain, is by a real value called the \textit{index} of the vector field. The index represents the number of times a vector field makes a complete counter-clockwise turn along a closed curve. If the vector field is smooth within the subregion, then the index of the vector field will be zero. Conversely, if the index of the vector field is non-zero, this implies the existence of discontinuities within the simple subregion. 
    \begin{definition}\label{definition index in terms of angles}
        Let the subregion $S_R \subset S$ be simple and with sufficiently smooth, positively oriented boundary $\partial S_R$, with parametrisation $\underline{\gamma}\in C([0,L],\partial S_R)$. For a given $\mathbf{u}\in \mathcal{A}_{\tan}$, I assume, without loss of generality, there is at most one point of discontinuity of $\mathbf{u}$ along $\partial S_R$
        \begin{equation}\label{chapter 2: index boundary integral parameterisation assumption}
            J_{\mathbf{u}}\cap \partial S_R = \emptyset \text{ or } J_{\mathbf{u}}\cap \partial S_R = \underline{\gamma}(0) = \underline{\gamma}(L).
        \end{equation}
        The index of $\mathbf{u}$ in $S_R$ is defined to be
        \begin{equation}\label{chapter 2: index defn}
            \textrm{Index}(\mathbf{u},S_R):= \dfrac{1}{2\pi}\int\limits_{0}^L\dfrac{d\phi}{ds}\ ds,
        \end{equation}
        where $\phi$ is the angle from $\mathbf{u}\circ \underline{\gamma}$ to $\mathbf{e}^1\circ \underline{\gamma}$.
    \end{definition}
    \begin{remark}
        A reader might ask the questions ``What if there were two singular values on the boundary of the simple domain? What if I want to know the index of a vector field on a non-simple domain?'' To answer both of those questions, I consider the work of Mermin, who investigated group theory properties of the index \cite{Defect_Group_theory}.
    
        For a given a finite sequence of non-intersecting, open, simple subregions, denoted $S^1_R,S_R^2, \ldots, S_R^N \subset S$, that satisfy equation (\ref{chapter 2: index boundary integral parameterisation assumption}). The index of the vector field in $\bigcup\limits_{i=1}^N S^i_R$ is the sum of the index of each subregion
        \begin{equation}\label{chapter 2: Summation of the index}
            \textrm{Index}\left(\mathbf{u},\bigcup\limits_{i=1}^N S^i_R\right) := \sum_{i=1}^N \textrm{Index}(\mathbf{u}, S^i_R).
        \end{equation}
    \end{remark}
    The index of a tangential vector field $\mathbf{u} \in \mathcal{A}_{\tan}$ may be expressed as an integral of the curvilinear components $u_1$ and $u_2$.
    \begin{lemma}\label{Surface index to 2D index}
        Let $\mathbf{u}\in \mathcal{A}_{\tan}$ have the decomposition given by equation (\ref{decomposition of a tangential field}) then I have that 
        \begin{equation}\nonumber
            \textrm{Index}(\mathbf{u}, S_R) = \dfrac{1}{2\pi}\oint\limits_{\partial S_R} u_1 du_2 - u_2 du_1.
        \end{equation}
    \end{lemma}
    \begin{proof}
        The definition of the algebraic value from equation (\ref{Definition of Algebraic Value}) and the decomposition of $\mathbf{u}$ from equation (\ref{decomposition of a tangential field}) implies that \begin{align}
            \left[\dfrac{D\mathbf{u}}{dt} \right] &= \dfrac{d\hat{\mathbf{u}}}{dt}\cdot \left(\hat{\mathbf{N}}\times \hat{\mathbf{u}} \right), \nonumber \\ 
            & = \hat{u}_1\left( \dfrac{d \hat{u}_2}{dt}+\hat{u}_1\dfrac{d \left( \mathbf{F}\circ \underline{\beta}\right)}{dt}\right) - \hat{u}_2\left(\dfrac{d\hat{u}_1}{dt} - \hat{u}_2 \dfrac{d \left( \mathbf{F}\circ \underline{\beta}\right)}{dt}\right), \nonumber \\ &= \hat{u}_1\dfrac{d \hat{u}_2}{dt} - \hat{u}_2\dfrac{d\hat{u}_1}{dt} + \dfrac{d \left( \mathbf{F}\circ \underline{\beta}\right)}{dt}. \nonumber
        \end{align}
        Comparing the above result with equation (\ref{Equation 1}), I immediately deduce that $$\dfrac{d\phi}{dt}=\hat{u}_1\dfrac{d \hat{u}_2}{dt} - \hat{u}_2\dfrac{d\hat{u}_1}{dt},$$ which, when substituted into the definition of Index from equation (\ref{chapter 2: index defn}), yields the result.
    \end{proof}
\section{Derivation of the generalisation of the Poincaré-Hopf Theorem}
\subsection{Gaussian curvature and the Gauss-Bonnet theorem}
    In this section, the generalisation of the Poincaré-Hopf Theorem will be proven, beginning with the derivation of the relation between the Gaussian curvature and the index of a vector field.
    \begin{lemma}
        Let $S_R\subset S$ be a simple subregion with a sufficiently smooth, positively oriented boundary $\partial S_R$, parametrised by $\underline{\gamma} \in C([0,L],\partial S_R)$. The vector field $\mathbf{u}\in \mathcal{A}_{\tan}$ satisfies
        \begin{equation}\label{Integral of Gaussian Curvature definition}
            \int\limits_{0}^L\left( \dfrac{d \phi}{dt} - \left[\dfrac{D\mathbf{u}}{dt} \right] \right) \ dt = \iint\limits_{S_R}K \ d\sigma,
        \end{equation}
        In the above equation, $\phi$ denotes the angle between $\mathbf{u} \circ \underline{\gamma}$ and $\mathbf{e}^1\circ \underline{\gamma}$, $K \in L^1(S)$ denotes the Gaussian curvature, defined in equation (\ref{Gaussian Curvature}), and $d \sigma := |\mathbf{x}^1||\mathbf{x}^2| d\underline{\omega}$ denotes the differential with respect to curvilinear coordinates.
    \end{lemma}
    \begin{proof}
        Consider the integral of equation (\ref{Equation 1}) over the interval $(0,L)$,
        \begin{equation}\label{temp 1}
            \int\limits_{0}^L\left( \dfrac{d \phi}{dt} - \left[\dfrac{D\mathbf{u}}{dt} \right] \right) \ dt = -\int\limits_{0}^L\underline{\beta}'\cdot (\mathbf{F}\circ \underline{\beta}) dt.
        \end{equation}
        Let the simple subregion $R \subset \Omega$ be defined by $R := \mathbf{x}^{-1}(S_R)$, with boundary $\partial R$ parametrised by $\underline{\beta}:=\mathbf{x}^-1\circ \underline{\gamma}$. Applying Green's Theorem and the definition of Gaussian curvature, from equation (\ref{Gaussian Curvature}), to equation (\ref{temp 1}) yields 
        \begin{equation}\nonumber
            \int\limits_{0}^L\left( \dfrac{d \phi}{dt} - \left[\dfrac{D\mathbf{u}}{dt} \right] \right) \ dt = -\iint\limits_{R}\left(\dfrac{\partial F_2}{\partial \omega_1}-\dfrac{\partial F_1}{\partial \omega_2} \right)d \underline{\omega} = \iint\limits_{R}K |\mathbf{x}^1||\mathbf{x}^2|d \underline{\omega}.
        \end{equation}
    \end{proof}
    The Gauss-Bonnet Theorem is a critical in proving the generalisation of the Poincaré-Hopf Theorem for simple surfaces. The Gauss-Bonnet Theorem states that a surface $S$, that satisfies the assumptions from sections (\ref{Surface assumptions})-(\ref{Boundary assumptions}), also satisfies the following \cite{Differential_Geo}
    \begin{equation}\nonumber %\label{Gauss-Bonnet theorem}
        2\pi \chi(S) = \oint\limits_{\partial S} kg(\underline{\gamma}) \ ds + \iint\limits_S K \ d\sigma + \sum_{i=1}^{\text{V}_{\partial S}} \tau_i,
    \end{equation}
    where $\underline{\gamma}:[0,L]\rightarrow \partial S$ is the parametrisation of $\partial S$, for $L >0$ the arclength of $\partial S$.
\subsection{Proof of the generalised Poincaré-Hopf for simple surfaces}
    \begin{theorem}[Generalised Poincaré-Hopf for simple surfaces]\label{Generalised Poincaré-Hopf for simple surfaces}
        Let $S_R\subset S$ be a simple subregion, with sufficiently smooth boundary $\partial S_R$ parametrised by $\underline{\gamma} \in C([0,L],\partial S_R)$. Let $\tau_1, \ldots, \tau_{V_{e,\partial S_R}} \in (-\pi,\pi)$ denote the exterior angles of $S_R$. Let $\mathbf{u}\in \mathcal{A}_{\tan}$ satisfy $J_{\mathbf{u}}\cap \partial S_R \in \{\emptyset,\{\underline{\gamma}(0)\}\}$. The field $\mathbf{u}$ satisfies
        \begin{equation}\label{Simple domain Poincare Hopf generalisation}
            2\pi = \sum_{i=1}^{\text{V}_{\partial S_R}} \tau_i + \oint\limits_{\partial S_R} \dfrac{d \theta}{dt}\ dt + 2\pi \textrm{Index}(\mathbf{u},S_R),
        \end{equation}
        where $\theta$ is the angle between $\mathbf{u}\circ \underline{\gamma}$ and the unit tangent to the boundary $\underline{\gamma}'$. 
    \end{theorem}
    \begin{remark}
        Although the Gauss-Bonnet Theorem can be applied to non-simple surfaces, in this section I shall focus on its application to simple surfaces. I shall use theorem \ref{Generalised Poincaré-Hopf for simple surfaces} with lemma \ref{Integral of Gaussian Curvature definition} to deduce a generalised Poincaré-Hopf for simple surfaces.  
    \end{remark}
    \begin{proof}
        The definition of a simple domain implies that it is homeomorphic to a disk, consequently it has Euler characteristic $\chi =1$. Applying the Gauss-Bonnet Theorem, to the surface $S_R$, and substituting equation (\ref{Integral of Gaussian Curvature definition}) yields
        \begin{align}
            2\pi &= \sum_{i=1}^{\text{V}_{\partial S_R}} \tau_i + \oint\limits_{\partial S_R} kg(\underline{\gamma}) \ dt + \iint\limits_{S_R} K \ d\sigma \nonumber \\ &= \sum_{i=1}^{\text{V}_{\partial S_R}} \tau_i + \int\limits_{0}^L \left( kg(\underline{\gamma}) - \left[\dfrac{D\mathbf{u}}{dt} \right] \right)\ dt + \int\limits_{0}^L\dfrac{d \phi}{dt} \ dt,
        \end{align}
        where $\phi$ is the angle between $\mathbf{u}\circ \underline{\gamma}$ and $\mathbf{e}^1\circ \underline{\gamma}$. The definition of the index of a vector field, from equation (\ref{definition index in terms of angles}), implies that 
        \begin{equation}\nonumber
            2\pi = \sum_{i=1}^{\text{V}_{\partial S_R}} \tau_i + \int\limits_{0}^L \left( kg(\underline{\gamma}) - \left[\dfrac{D\mathbf{u}}{dt} \right] \right)\ dt + 2\pi \textrm{Index}(\mathbf{u},S_R).
        \end{equation}
        As $S_R$ being parametrised by its arclength, this implies that $|\underline{\gamma}'|=1$. Additionally, the geodesic curvature can be expressed as the algebraic value of $\underline{\gamma}'$. Thus, equation (\ref{angle between two vector fields}) implies that 
        \begin{equation}\label{chapter 2: defn of theta}
            kg(\underline{\gamma}) - \left[\dfrac{D\mathbf{u}}{dt} \right] = \left[\dfrac{D\underline{\gamma}'}{dt} \right] - \left[\dfrac{D\mathbf{u}}{dt} \right] = \dfrac{d\theta}{dt},
        \end{equation}
        where $\theta$ is the angle between $\mathbf{u}\circ \underline{\gamma}$ and $\underline{\gamma}'$.
    \end{proof}
\subsection{Proof of the generalised Poincaré-Hopf for non-simple surfaces}
    \begin{theorem}[Generalised Poincaré-Hopf for non-simple surfaces]
        For a surface $S$, the vector field $\mathbf{u}\in \mathcal{A}_{\tan}$ satisfies the following conservation law,
        \begin{equation}\label{Conservation law for defects}
            \chi(S) = \sum_{k=1}^{\text{V}_{\partial S}} \dfrac{\tau_k}{2\pi} + \dfrac{1}{2\pi}\oint\limits_{\partial S} \dfrac{d\theta}{ds}\ ds + \rm{Index}(\mathbf{u},\mathit{S}),
        \end{equation}
        where:
        \begin{itemize}
            \item $\chi$ is the Euler Characteristic, given in equation (\ref{chapter 2: Euler Characteristic}).
            \item $\theta$ is the angle between $\mathbf{u}$ and the tangent to the boundary, given in equation (\ref{chapter 2: defn of theta}).
            \item $\tau_1, \ldots, \tau_{\text{V}_{\partial S}}$ are the exterior angles of $S$, from Section (\ref{Boundary assumptions}).
        \end{itemize}
    \end{theorem}
    \begin{remark}
        In the previous subsection, the generalisation was derived for non-simple surfaces. To expand this result to non-simple surfaces, the surface $S$ will be triangulated and equation (\ref{Simple domain Poincare Hopf generalisation}) will be applied to each triangular component. Additionally, comparing this result to the Morse index formula in equation (\ref{chapter 2: Morse index formula}) it is clear that
        \begin{equation}
            \text{ind}(P_{\partial S}\mathbf{U},\partial_{-}S) = \dfrac{1}{2\pi}\oint\limits_{\partial S} \dfrac{d\theta}{ds}\ ds + \sum_{k=1}^{\text{V}_{\partial S}} \dfrac{\tau_k}{2\pi}, \quad \text{ind}(\mathbf{U},S) = \text{Index}(\mathbf{u},S).
        \end{equation}
        However, Morse's formula did not permit singular values on the boundary of the surface, which my formula accounts for.
    \end{remark}
\begin{proof}
    Let $\mathcal{T}:=\{T_i\}_{i=1}^{\text{F}}$ denote the triangulation of the surface $S$, see Definition \ref{Triangulation of a surface}. For each triangle $T_i \subset S$ let:
    \begin{itemize}
        \item $\mathcal{E}_{T_i}=\{e_{ij}\}_{j=1}^3 \subset S$ denote the set of edges of $T_i$.
        \item $\{\tau_{ij}\}_{j=1}^3 \subset (-\pi,\pi)\setminus \{0\}$ denote the set of exterior angles of $T_i$.
        \item $\underline{\gamma}_i \in C([0,L_i],\partial T_i)$ parametrise the boundary $\partial T_i$; without loss of generality, it is assumed that $\underline{\gamma}_i$ is parametrised by its arclength and is positively oriented.
    \end{itemize}
    For a given $\mathbf{u}\in \mathcal{A}_{\tan}$ it is assumed that the triangulation satisfies
    \begin{equation}\nonumber
        \partial T_i \cap J_{\mathbf{u}} = \emptyset, \text{ or }\partial T_i \cap J_{\mathbf{u}} = \underline{\gamma}_i(0)=\underline{\gamma}_i(L_i)\quad \forall i = 1,\ldots, \text{F}.
    \end{equation}
    For this section, the following notation shall be implemented
    \begin{align}
        \text{E} = & \text{ number of edges of }\mathcal{T}\nonumber \\
        \text{V} = & \text{ number of vertices of }\mathcal{T}\nonumber \\
        \text{E}_{\partial S}               =& \text{ number of external edges of }\mathcal{T}\nonumber \\
        \text{E}_{S}                        =& \text{ number of internal edges of }\mathcal{T} \text{ (ignoring repeated edges)}.\nonumber \\
        \text{V}_{\partial \mathcal{T}}     =& \text{ number of external vertices of }\mathcal{T}.\nonumber \\
        \text{V}_{\partial S}               =& \text{ number of vertices of }\partial S. \nonumber \\
        \text{V}_{\partial \mathcal{T} / S} =& \text{ number of vertices of the triangulation which are not vertices of }\partial S. \nonumber \\
        \text{V}_{\mathcal{T}}              =& \text{ number of internal vertices of }\mathcal{T}.\nonumber \\
        \sigma_{ij}                         :=& \pi - \tau_{ij}, \text{ the interior angles of the triangle }T_i. \nonumber
    \end{align}

        The vector field $\mathbf{u} \in \mathcal{A}_{\tan}$ satisfies
        \begin{align}
            2\pi &= \sum_{j=1}^3\tau_{ij} + \sum_{j=1}^3\int\limits_{e_{ij}} \dfrac{d\theta_i}{dt}dt + 2\pi \textrm{Index}(\mathbf{u},T_i),\nonumber \\
            2\pi \text{F} &= \sum_{i=1}^{\text{F}}\sum_{j=1}^3 \tau_{ij} + \sum_{i=1}^{\text{F}}\sum_{j=1}^3 \int\limits_{e_{ij}} \dfrac{d\theta_i}{dt}dt + 2\pi \sum_{i=1}^{\text{F}}\textrm{Index}(\mathbf{u},T_i).\label{Appendix: Faces summation}
        \end{align}
        where $\theta_i$ is the angle between $\mathbf{u}\circ \underline{\gamma}_i$ and $\underline{\gamma}_i'$. Consider a pair of distinct triangles $T_{i}$ and $T_{k}$ that share an edge $e_{in} = e_{km}$, for $n,m \in \{1,2,3\}$. As the boundary of each triangle is positively oriented, the following identity is obtained
        \begin{equation}\label{appendix: angle integral}
            \int\limits_{e_{in}} \dfrac{d\theta_{i}}{dt}dt = - \int\limits_{e_{km}} \dfrac{d\theta_{k}}{dt}dt.
        \end{equation}
        The only edges which are not shared between triangles are ones which intersect the boundary $\partial S$, and consequently
        \begin{equation}\nonumber
            \sum_{i=1}^{\text{F}}\sum_{j=1}^3 \int\limits_{e_{ij}} \dfrac{d\theta_i}{dt}dt = \oint\limits_{\partial S} \dfrac{d\theta}{dt}dt,
        \end{equation}
        As a consequence of the above and the summation properties of the index, from equation (\ref{chapter 2: Summation of the index}), equation (\ref{Appendix: Faces summation}) is simplified to 
        \begin{equation}\nonumber
            2\pi \text{F} = \sum_{i=1}^{\text{F}}\sum_{j=1}^3 \tau_{ij} + \oint\limits_{\partial S} \dfrac{d\theta}{dt}dt + 2\pi\textrm{Index}(\mathbf{u},S)
        \end{equation}
        Consequently, the only term left to simplify is the sum of the exterior angles; it is clear that
        \begin{equation}\nonumber 
            \text{V}_{\partial \mathcal{T}/S} + \text{V}_{\partial S} = \text{V}_{\partial \mathcal{T}}, \quad \sum_{i=1}^{\text{F}}\sum_{j=1}^3 \tau_{ij} = \sum_{i=1}^{\text{F}}\sum_{j=1}^3 (\pi - \sigma_{ij}) = 3\pi \text{F} - \sum_{i=1}^{\text{F}}\sum_{j=1}^3 \sigma_{ij}.
        \end{equation}
        As each triangle has three edges, the sum of all edges (including duplicate edges) is equal to three times the number of triangles, $3\text{F} = 2\text{E}_S + \text{E}_{\partial S}$. Consequently, 
        \begin{equation}\label{appendix: angle sum intermediate}
            \sum_{i=1}^{\text{F}}\sum_{j=1}^3 \tau_{ij} = 2\pi \text{E}_S + \pi \text{E}_{\partial S} - \sum_{i=1}^{\text{F}}\sum_{j=1}^3 \sigma_{ij}.
        \end{equation}
        If a vertex is internal, then the sum of internal angles at that vertex is $2\pi$; if a vertex is external but not a vertex of $\partial S$, then the sum of internal angles at that vertex is $\pi$. Recalling that $\{\tau_k\}_{k=1}^{\text{V}_{\partial S}}$ is the set of external angles of $\partial S$, it is clear that
        \begin{equation}\nonumber
            \sum_{i=1}^{\text{F}}\sum_{j=1}^3 \sigma_{ij} = 2\pi \text{V}_{\mathcal{T}} + \pi \text{V}_{\partial \mathcal{T}/ S} + \sum_{k=1}^{\text{V}_{\partial S}} (\pi-\tau_k).
        \end{equation}
        As the boundary curves are closed, it is clear that $\text{E}_{\partial S} = \text{V}_{\partial \mathcal{T}}$ and $\text{E}_{\partial S} = 2\text{E}_{\partial S} - \text{V}_{\partial \mathcal{T}}$. Hence
        \begin{equation}\label{appendix: sum of exterior angles tessaltion}
            \sum_{i=1}^{\text{F}}\sum_{j=1}^3 \tau_{ij} = 2\pi (\text{E}_S+ \text{E}_{\partial S}) - 2\pi (\text{V}_{\mathcal{T}} + \text{V}_{\partial \mathcal{T}}) + \sum_{k=1}^{\text{V}_{\partial S}} \tau_k.
        \end{equation}
        It is clear that $\text{E} = \text{E}_S+ \text{E}_{\partial S}$ and $\text{V} = \text{V}_{\mathcal{T}}+ \text{V}_{\partial\mathcal{T}}$. Consequently, substituting equations (\ref{appendix: angle integral}) and (\ref{appendix: sum of exterior angles tessaltion}) into (\ref{Appendix: Faces summation}) yields
        \begin{equation}\nonumber
            2\pi(\text{V}-\text{E}+\text{F}) = \sum_{k=1}^{\text{V}_{\partial S}} \tau_k + \oint\limits_{\partial S} \dfrac{d\theta}{ds}\ ds + 2\pi \textrm{Index}(\mathbf{u},S).
        \end{equation}
        The definition of the Euler characteristic from equation (\ref{chapter 2: Euler Characteristic}) implies the result.
    \end{proof}
\section{Conclusion and future work}
    In this paper, I derived a generalisation of the Poincaré-Hopf Theorem for surfaces with non-smooth line boundaries and for vector fields with arbitrary boundary conditions. Although the theorem has been described as a generalisation, a number of assumptions can be relaxed in future explorations of this topic. One assumption that could be relaxed, is the assumption that the manifold must be expressed as a single smooth chart $\mathbf{x}^{-1}:S \rightarrow \Omega$, rather than an atlas of charts \cite{JostJurgen2013RGaG}. Future work could focus on deriving a similar result for a given atlas; although I hypothesise that the equation (\ref{Conservation law for defects}) would be applicable to such surfaces. The proof would most likely involve applying equation (\ref{Conservation law for defects}) on each chart and then simplifying the sum of angles in a similar manner to the one found in equation (\ref{appendix: angle sum intermediate}).
  
    A second direction for future research could be the derivation of a conservation law for the indices of an analytic vector field in domains of higher dimension. Poincaré derived the theorem for two-dimensional vector fields \cite{Poincare_original} and Hopf generalised the result for higher dimensions \cite{Hopf_original}. Additionally, the index of an analytic vector field, in higher dimensions, can be defined and explicitly calculated \cite{Brouwer_Degree}. The goal of this future research would be to derive a similar result without the assumption that the vector field must be parallel to the boundary of the manifold.
  
    A third direction for future research could be to remove the assumption that the zeroes of the unit vector field are isolated. The index of a vector field and the Poincaré–Hopf Theorem may be extended for vector fields with non-isolated zeroes \cite{BrasseletJean-PaulVfoS}.
\section*{Acknowledgements}
I am grateful for the financial support of the University of Bath's department of Mathematical Sciences. Additionally, I am thankful for the support and guidance of Prof. Kirill Cherednichenko, Prof. Apala Majumdar, and Prof. Jeybal Sivaloganathan, whose guidance in the aspects of functional analysis and liquid crystals was invaluable. 
\printbibliography
\end{document}